\documentclass[10pt]{amsart}
\usepackage{amssymb}
\usepackage{amscd}

\newcommand{\bZ}{{\mathbb Z}}
\newcommand{\bR}{{\mathbb R}}

\newcommand{\bC}{{\mathbb C}}
\newcommand{\bF}{{\mathbb F}}

\newcommand{\bT}{{\mathbb T}}

\newcommand{\bG}{{\mathbb G}}

\newcommand{\II}{{I_{\infty}^2}}

\newtheorem{thm}{Theorem}[section]
\newtheorem{lemma}[thm]{Lemma}
\newtheorem{cor}[thm]{Corollary}

\newtheorem{conj}[thm]{Conjecture}

\numberwithin{equation}{section}

\begin{document}

\title[gamma filtration  of flag varieties ]{The gamma filtrations for  the spin groups }
 
\author{Nobuaki Yagita}

\address{ faculty of Education, 
Ibaraki University,
Mito, Ibaraki, Japan}
 
\email{nobuaki.yagita.math@vc.ibaraki.ac.jp, }


\maketitle

\begin{abstract}
Let $G$ be a compact Lie group and $T$ its maximal torus.
In this paper, we try to  compute $gr_{\gamma}^*(G/T)$
the graded ring associated  with the gamma filtration of
the complex $K$-theory $K^0(G/T)$ for $G=Spin(n)$.
In particular,  we give a counterexample for a conjecture by Karpenko  
when $G=Spin(17)$. The arguments for $E_7$
 in $\S 11$ of the old version were not correct, and they are deleted in this version.

\end{abstract}

\section{Introduction}

Let us fix a prime number $p$, and assume all theories
($H^*(X), CH^*(X), K^*(X)$) are $p$-localized theories.
 Let $G$ and  $T$ be a simply  connected compact Lie group
and its maximal torus. 
Given a field $k$ with $ch(k)=0$,
let $G_k$ and $T_k$ be the split  reductive group and 
a split maximal torus over the field $k$,
corresponding to $G$ and  $T$.  Let $B_k$ be the Borel
subgroup containing $T_k$.
Let $\bG$ be a $G_k$-torsor.  Then $\bF=\bG/B_k$ is a
twisted form of the flag variety $G_k/B_k$. 

In this paper, we study the  graded ring $gr_{\gamma}(\bF)$ associated to the $\gamma$-filtration of the algebraic $K$-theory $K_{alg}^0(\bF)$, mainly for
$G=Spin(odd)$ and $p=2$. 
Since $K_{alg}^0(\bF)$ is isomorphic to the topological 
$K$-theory $ K_{top}^0(G/T)$ [Pa], we know $gr_{\gamma}(\bF)\cong
gr_{\gamma}^*(G/T)$ for the algebraic and topological $\gamma$-filtrations.

Moreover we take 
$\bG$  a versal $G_k$-torsor i.e.,
$\bG$ is isomorphic to $G_{k(S)}$-torsor over
$Spec(k(S))$ given by the generic fiber of $GL_N\to S$ where $S=GL_N/G_k$ for an embedding 
$G\subset GL_N$ ( for the properties of a versal $G_k$-torsor, see $\S 3$ below or
see \cite{Ga-Me-Se}, \cite{To2}, \cite{Me-Ne-Za}, \cite{Ka1}).  
When $\bG/B_k=\bF$ is
 a versal  flag variety, it is known  (\cite{Me-Ne-Za}, \cite{Ka1}) that $CH^*(\bF)$ is generated by Chern classes.  Hence we know  $gr_{\gamma}^*(\bF)\cong gr^*_{geo}(\bF)$; the graded ring associated the geometric filtration (defined from degree of $CH^*(X)$)
 (see Chapter 15 in [To3], \cite{YaG}). 
\begin{lemma}
Let $G$ be a compact simply connected Lie group.
Let $\bF=\bG/B_k$ be the versal flag variety. Then 
\[ gr_{\gamma}^*(G/T)\cong
gr_{geo}^*(\bF)\cong CH^*(\bF)/I(1)
\quad for\ some \ ideal\ I(1)\subset CH^*(\bF).\]
\end{lemma} 

  Karpenko conjectured that 
the above $I(1)=0$
(\cite{Ka1}, \cite{Ka2}).
In this paper, we try to compute $gr_{\gamma}^*(G/T)$
for $G=Spin(2\ell+1)$. In particular, we see $I(1)\not =0$
for $G=Spin(17)$.

For these computations, we use the generalized Rost
motive $R(\bG)$ defined by Petrov-Semenov-Zainoulline
\cite{Pe-Se-Za} such that the ($p$-localized) motive $M(\bF)$ of the flag variety $\bF$ is written
\[ M(\bF)\cong R(\bG)\otimes (\oplus_s \bT^{\otimes s}),
\quad for\ the \ Tate \ motive \ \bT.\]

For $G=Spin(2\ell+1)$, $G''=Spin(2\ell+2)$, we know
$R(\bG)\cong R(\bG'')$ from  a theorem
by Vishik-Zainoulline \cite{Vi-Za}.  So in this paper,
we only study $Spin(odd)$ but many  arguments for
$Spin(even)$ are similar.

For $\ell\le 2$,  $H^*(G)$
has no torsion. For $\ell=3,4$, $gr_{\gamma}^*(\bF)$
are  given in \cite{YaC} and $I(1)=0$.
For $\ell=5$, $gr_{\gamma}^*(\bF)$ is given in \cite{YaG},
\cite{Ka2}. Moreover Karpenko shows $I(1)=0$ and this
Karpenko result is the first example that
$R(\bG)$ is not the original Rost motive, but $CH^*(R(\bG))$ is given.
For $\ell=6$, $gr^*_{\gamma}(\bF)$ is given $\S 9$ in this paper, while we can not see $I(1)=0$ or not. 
For $\ell\ge 7$,  very partial results are only given here.
However, we see $I(1)\not =0$ for $\ell=8$.

  Let us write by $\Lambda(a,...,b)$ the exterior algebra
over $\bZ/2$ generated by $a,...,b$.
We note the following fact 
\begin{lemma} Let $G=Spin(2\ell+1)$ and $\bG$ be versal.  Let $2^t\le \ell<2^{t+1}$. Then 
there are elements $c_2,...,c_{ \ell}$ and
$e_{2^{t+1}}$ with $deg(c_i)=i$, $deg(e_j)=j$ such that there is an additive  surjection
\[ \Lambda(c_2,...,c_{\ell})\otimes 
\bZ/2[e_{2^{t+1}}]\twoheadrightarrow 
CH^*(R(\bG))/2.\]
\end{lemma}
Here $c_i$ is represented by some  $i$-th Chern class.
For a ring $A$, let us write by $A\{a,...,b\}$ the $A$-free module generated by $a,...,b$.  For example, we 
can write 
\begin{thm} Let $G=Spin(13)$
and $\bG$ be versal.  Then $gr_{\gamma}^*(R(\bG))$ is additively generated by  products of $1,c_2,c_3,...,c_6$ and $e_8$ as follows
\[ gr_{\gamma}^*(R(\bG))\cong 
    A\otimes (\bZ_{(2)}\{1,c_6\}\oplus \bZ/2\{c_4\})\oplus
\bZ/2\{c_6c_4\}\]
where $A= \bZ_{(2)}\{1,c_3,c_5,e_8\}\oplus \bZ/2\{c_2\}.$
\end{thm}

When $G=Spin(17)$, elements in 
the Chow ring  
$CH^*(R(\bG))/2$ are additively generated by  products of
$1,c_2,...,c_7$. 
 Moreover the torsion index $t(G)=2^4$.
Then we can  see for
$x=c_2c_3c_6c_7,$
\[ x=0 \in gr_{\gamma}^*(R(\bG))/2,\ \  
  but \ \  x\not =0 \in gr(2)^*(R(\bG))/2.\]  Here $gr(n)^*(X)$ is the graded
associated ring for the integral Morava $\tilde K(n)$-theory, so that
\[gr(n)^*(X)\cong CH^*(X)/I(n)\quad  for\  some\ ideal\ I(n),\]
and $gr(1)^*(X)\cong gr_{geo}^*(X)$.
Hence this $x$ is nonzero  in $CH^*(R(\bG))/2$, i.e., $0\not =x\in I(1)$.

{\bf Remark.}
The arguments (for $(G,p)=(E_7,2)$) in $\S 11$ and 
Theorem 1.4 of the old version were not correct, and they are deleted in this version.

The plan of this paper is the following.
In $\S 2$, we recall and prepare the topological arguments.
 In $\S 3$, we recall the decomposition
of the motive of a versal flag variety.
In $\S 4$, we recall the filtrations of the $K$-theory.
In $\S 5$, we study  the connective $K$-theory   $k^*(\bar R(\bG))$.
In $\S 6$, we study the general properties when
$G=Spin(2\ell+1)$.
In $\S 7,\ \S 8$,  we study cases $\ell=3$ (and $4$),
$\ell=5$  respectively. 
In $\S 9$, we study the case $\ell=6$.
In $\S 10$ we study the 
cases $\ell\ge 7$ and see the counterexample when $\ell=8$. In $\S 11$, we note the cases $\ell$ large.

 \section{Lie groups $G$ and the  flag manifolds $G/T$}  

Let $G$ be a connected 
 compact Lie group.
By Borel (\cite{Bo}, \cite{Mi-Tod}), its $mod(p)$ cohomology is written as 
\[ (*)\quad gr H^*(G;\bZ/p)\cong P(y)/p\otimes \Lambda(x_1,...,x_{\ell}),
\quad \ell=rank(G)\]
\[ \quad  with  \ \  P(y)=\bZ_{(p)}[y_1,...,y_k]/(y_1^{p^{r_1}},...,
y_k^{p^{r_k}})
\]
where the degree $|y_i|$ of $y_i$ is even and 
$|x_j|$ is odd.  In this paper, we only consider 
this $grH^*(G;\bZ/p)$ in $(*)$, and write it simply $H^*(G;\bZ/p)$.

 Let $T$ be the  maximal torus of $G$ and $BT$ be the classifying space of $T$.
We consider the fibering (\cite{Tod}, \cite{Mi-Ni})
$ G\stackrel{\pi}{\to}G/T\stackrel{i}{\to}BT$
and the induced spectral sequence 
\[ E_2^{*,*'}=H^*(BT;H^{*'}(G;\bZ/p)) \Longrightarrow H^*(G/T;\bZ/p).\] 
The cohomology of the classifying space of the torus is  given by
$H^*(BT)\cong S(t)=\bZ[t_1,...,t_{\ell}]$ with $|t_i|=2$,
where $t_i=pr_i^*(c_1)$ is the $1$-st  Chern class
induced from
$ T=S^1\times ...\times S^1\stackrel{pr_i}{\to}S^1\subset U(1)$
for the $i$-th projection $pr_i$.
Note that $\ell=rank(G)$ is also the number of the odd degree generators $x_i$ in
$H^*(G;\bZ/p)$.  

It is well
known that $y_i$ are permanent cycles (i.e., exist in $E_{\infty}^{0,*'}$) and 
that there is a regular sequence (\cite{Tod}, \cite{Mi-Ni})
$(\bar b_1,...,\bar b_{\ell})$ in $H^*(BT)/(p)$ such that $d_{|x_i|+1}(x_i)=\bar b_i$ (this $\bar b_i$ is called the transgressive element).  Thus we get
\[ E_{\infty}^{*,*'}\cong grH^*(G/T;\bZ/p)\cong P(y)/p\otimes 
S(t)/(\bar b_1,...,\bar b_{\ell}).\]
 Moreover we know that $G/T$ is a manifold 
such that $H^*(G/T)$ is torsion free, and 
\[grH^*(G/T)\cong P(y)\otimes S(t)/(b_1,...,b_{\ell})
\quad where \ b_i=\bar b_i\ mod(p).\]
For ease of he notation, let us write simply
$S(t)/(b)=S(t)/(b_1,...,b_{\ell})$.

Let $BP^*(-)$ be the Brown-Peterson theory
with the coefficients ring $BP^*\cong \bZ_{(p)}[v_1,v_2,...],$
$|v_i|=-2(p^i-1)$ (\cite{Ha}, \cite{Ra}).
Since $H^*(G/T)$ is torsion free,  the 
Atiyah-Hirzebruch 
spectral sequence (AHss) collapses.  Hence  we also know 
$gr BP^*(G/T)\cong BP^*\otimes grH^*(G/T)$.

Recall the  Morava $K$-theory
$K(n)^*(X)$ with the coefficient ring $K(n)^*=\bZ/p[v_n,v_n^{-1}]
$.
Similarly we can define connected or integral Morava
K-theories with 
\[ k(n)^*=\bZ/p[v_n],\ \ \tilde k(n)^*=\bZ_{(p)}[v_n],
\ \ \tilde K(n)^*=\bZ_{(p)}[v_n,v_n^{-1}]. \]
It is known that
$\tilde K(1)^*(X)\cong  BP^*(X)\otimes _{BP^*}\tilde K(1)^*$ as the Conner-Floyd theorem. The relation to the usual $K$-theory $K^*(X)$
is given using the Bott periodicity $B$ 
\[ K^*(X)\cong \tilde K(1)^*(X)\otimes_{\tilde K(1)^*}\bZ_{(p)}[B,B^{-1}],\quad with\
B^{p-1}=v_1.\]

We recall the Milnor operation
$ Q_i: H^*(X;\bZ/p)\to H^{*+2p^i-1}(X;\bZ/p)$,
which 
is defined by $Q_0=\beta$ and $Q_{i+1}=P^{p^i}Q_i-Q_iP^{p^i}$
for the Bockstein operation $\beta$ and the reduced power operation $P^j$.
This $Q_i$-operation relates to the 
Morava $k(i)^*$-theory.
Recall the fibering $G\stackrel{\pi}{\to}G/T\to BT$.
Let us write  $v_0=p$ and $k(0)^*(X)=H^*(X)$.
Let $I_{\infty}=(p,v_1,...)$ be the invariant prime ideal in $BP^*$. Then we can prove
\begin{cor} (\cite{YaC}) 
Let $d_{|x|+1}(x)=b\not =0\in H^*(G/T)/p$ for $x\in H^*(G;\bZ/p)$.  Then there is a lift $b\in BP^*(BT)\cong BP^*\otimes S(t)$ of $b\in S(t)/p$ such that
\[ b=\sum_{i=0} v_iy(i) \in \ BP^*(G/T)/\II
\quad (i.e., \ b=v_iy(i)\in\ k(i)^*(G/T))/(v_i^2))\]
where $y(i)\in H^*(G/T;\bZ/p)$ with $\pi^*y(i)=Q_ix$.
\end{cor}

\section{versal flag varieties}

 Let $G_k$ be the split reductive algebraic  group corresponding to $G$, and $T_k$ be the split maximal
torus corresponding to $T$.  Let $B_k$ be the Borel subgroup
with $T_k\subset B_k$.   Note that $G_k/B_k$ is cellular, and
$CH^*(G_k/T_k)\cong CH^*(G_k/B_k)$.
Hence we have   
$ CH^*(G_k/B_k)\cong H^{2*}(G/T)$
and $CH^*(BB_k)\cong H^{2*}(BT).$

{\bf Remark.}  In this paper, for $x\in CH^*(X)$,
degree $|x|$ means two times of the usual degree for 
the Chow ring i.e., $|x|=2*$.

Let us write by $\Omega^*(X)$  the $BP$-version of the algebraic cobordism defined by Levine-Morel 
(\cite{Le-Mo},\cite{YaA},\cite{YaB})
  such that
\[ \Omega^*(X)=MGL^{2*,*}(X)_{(p)}\otimes_{MU_{(p)}^*}BP^*,
\quad \Omega^*(X)\otimes _{BP^*}\bZ_{(p)}\cong CH^*(X)\] where $MGL^{*,*'}(X)$ is the algebraic cobordism theory
defined by Voevodsky with $MGL^{2*,*}(pt.)\cong 
MU^*$ the complex cobordism ring.
There is a natural (realization) map $\Omega^*(X)\to
BP^*(X(\bC))$.  In particular, we have 
$\Omega^*(G_k/B_k)\cong BP^*(G/T).$
Recall  $I_n=(p,v_1,...,v_{n-1})$ and we  also note
\[ \Omega^*(G_k/B_k)/I_{\infty}\cong
CH^*(G_k/B_k)/p\cong H^*(G/T)/p.\] 

Let $\bG$ be a nontrivial $G_k$-torsor.
We can construct a twisted form of $\bG_k/B_k$ 
by $(\bG\times G_k/B_k)/G_k\cong \bG/B_k.$
We will study  the twisted flag variety $\bF=\bG/B_k$.

Let us consider an embedding of $G_k$ into the general linear group $GL_N$ for some large  $N$.  This makes $GL_N$ a $G_k$-torsor over the quotient variety $S=GL_N/G_k$.
Let $F$ be the function field $k(S)$ and  define
the $versal$ $G_k$-$torsor$ $E$ to be the $G_k$-torsor over $F$ given by the generic fiber of $GL_N\to S$. 
(For details, see \cite{Ga-Me-Se}, \cite{To2}, \cite{Me-Ne-Za}, \cite{Ka1}.) 

The corresponding flag variety $E/B_k$ is called the 
$versal$ flag
variety, which is considered as the most complicated twisted
flag variety (for given $G_k$). It is known that the Chow ring
$CH^*(E/B_k)$ is not dependent to the choice
of  generic $G_k$-torsors $E$ (Remark 2.3 in \cite{Ka1}).

Karpenko proves the following result
 for a versal flag variety. 
\begin{thm}
(Karpenko Lemma 2.1 in \cite{Ka1}, \cite{Me-Ne-Za})
Let  $\bG/B_k$ be a versal flag variety.  
Let $h^*(-)$ be an oriented cohomology theory
e.g., $h^*(X)=CH^*(C), K^*(X),\Omega^*(X)$.
Then the natural map
$h^*(BB_k)\to h^*(\bG/B_k)$ is surjective.
\end{thm}
\begin{cor}
If $\bG$ is versal, then $h^*(\bF)=h^*(\bG/B_k)$
is multiplicatively generated by elements $t_i$ in $S(t)$. 
\end{cor}

The versal case  of the main result in Petrov-Semenov-Zainoulline \cite{Pe-Se-Za} is
given  as
\begin{thm} (Theorem 5.13 in \cite{Pe-Se-Za}
and Theorem 4.3 in \cite{Se-Zh})
Let $\bG$ be a versal $G_k$-torsor, and   $\bF=\bG/B_k$. 
Then there is a $p$-localized  motive $R(\bG)$ such that
the ($p$-localized) motive $M(\bF)$ of the variety $\bF$ is decomposed
\[M(\bF)_{(p)}\cong  R(\bG)\otimes T\quad with\ T=(\oplus _u
 \bT^{\otimes u}).\]
Here $\bT^{\otimes u}$ are Tate motives with 
$CH^*(T)/p\cong  S(t)/(p,b)=S(t)/(p,b_1,...,b_{\ell}).$   Hence
\[  CH^*(\bF)/p\cong 
CH^*( R(\bG))\otimes S(t)/(p,b).\]
For  $\bar R(
\bG)=R(\bG)\otimes \bar k$, we have 
$ CH^*(\bar R(\bG))/p\cong P(y)/p$.
\end{thm}

{\bf Remark.}  In this paper, 
 $A\cong B$ for rings $A,B$
means a ring isomorphism.
 However $CH^*(R(\bG))$ does not have a
 canonical ring structure.
 Hence an isomorphism $A\cong CH^*(R(\bG))$
means only a (graded) additive isomorphism. 

Hereafter in this section, we always assume that
$\bG$ (and hence $\bF$) is versal.

Hence we have surjections for a  variety $\bF$
\[ CH^*(BB_k)\twoheadrightarrow  CH^*(\bF)\stackrel{pr.}{\twoheadrightarrow} CH^*(R(\bG)).\]
We study in \cite{YaC}
what elements in $CH^*(BB_k)$ 
generate $CH^*(R(\bG))$.

For ease of notations, let us write
$A(b)=\bZ/p[b_1,...,b_{\ell}]$.
By giving the filtration on $S(t)$ by $b_i$, we 
can write 
$gr S(t)/p\cong A(b)\otimes S(t)/(b).$

In particular, we have maps
$ A(b)\stackrel{i_A}{\to} CH^*(\bF)/p\to CH^*(R(\bG))/p.$
We also see that
the above composition map is surjective
(see also Lemma 3.6  below).
\begin{lemma}
Suppose that there are $f_1(b),...,f_s(b)\in A(b)$ such that 
$CH^*(R(\bG))/p\cong A(b)/(f_1(b),...,f_s(b))$. Moreover if $f_i(b)=0$ for $1\le i\le s$ 
also in $CH^*(\bF)/p$, we have the isomorphism
\[CH^*(\bF)/p\cong S(t)/(p, f_1(b),...,f_s(b)).\]
\end{lemma}

\begin{lemma}
Let $pr:CH^*(\bF)/p\to CH^*(R(\bG))/p$, and 
 $0\not =x\in Ker(pr)$.  Then  
$ x=\sum b't'$ with $b'\in A(b),$ $0\not =t'\in S(t)^+/(p,b)$
 i.e., $ |t'|>0.$
\end{lemma}
Let us write 
\[  y_{top}=\Pi_{i=1}^s y_i^{p^{r_i}-1}\quad (resp.\ t_{top})\]
the generator of the highest  degree 
in $P(y)$ (resp. $S(t)/(b)$) so that $f=y_{top}t_{top}$
is the fundamental class in $H^{2d}(G/T)$
for $2d=dim_{\bR}(G/T)$.
For $N>0$, let us  write \ \ 
\[A_N=\bZ/p\{b_{i_1}...b_{i_k}| |b_{i_1}|+...+|b_{i_k}|\le N\}
\subset A(b).  \]
\begin{lemma}  For $N=|y_{top}|$, the map
$ A_N\to CH^*(R(\bG))/p$ is surjective.
\end{lemma}
\begin{proof}  In the preceding lemma,  $A_{N}\otimes t'$ for $|t'|>0$ maps zero in $CH^*(R(\bG))/p$.  Since each element
in $S(t)$ is written by an element in $A_N\otimes S(t)/(b)$,
we have the lemma.
\end{proof}
\begin{cor}
If  $b_i\not=0$ in $CH^*(X)/p$,  then  so in $CH^*(R(\bG_k))/p$.
\end{cor}
\begin{proof}
Let $pr(b_i)=0$.  From Lemma 3.6, 
 $b_i=\sum b't'$ for $|t'|>0$, and hence $b'\in Ideal(b_1,...,b_{i-1})$.
This contradict to that $(b_1,...,b_{\ell})$ is regular.
\end{proof}

Now we consider the torsion index $t(G)$.
Let $dim_{\bR}(G/T)=2d$.  Then the torsion index is defined as
\[ t(G)=|H^{2d}(G/T;\bZ)/i^*H^{2d}(BT;\bZ)|
\quad where\ i:G/T\to BT.\]
Let $n(\bG)$ be the greatest common divisor of the degrees of all finite field extension $k'$ of $k$ such that $\bG$ 
becomes trivial over $k'$.  Then by Grothendieck 
\cite{Gr}, it is known that $n(\bG)$ divides $t(G)$.  Moreover, $\bG$ is versal, then $n(\bG)=t(G)$   (\cite{To2}, \cite{Me-Ne-Za}, \cite{Ka1}), which implies
that each element in $t(G)P(y)$ is represented by element in $ CH^*(BB_k)$.

It is well known that  if $H^*(G)$ has a $p$-torsion, then
$p$ divides the torsion index $t(G)$. Torsion index for
simply connected compact Lie groups are completely determined by Totaro \cite{To1}, \cite{To2}.

\begin{lemma} (\cite{YaC})
Let $\tilde b=b_{i_1}... b_{i_k}$ in
$S(t)$  such that
in $H^*(G/T)$
\[ \tilde b=p^s(y_{top}+\sum yt),\quad
|t|>0\]
for some $y\in P(y)$  and $t\in S(t)^+$.
Then $t(G)_{(p)}\le p^s$. 
\end{lemma}

\section{Filtrations of $K$-theories}

In this section, we recall the properties of
the graded rings $gr_{top}^*(X)$, $gr_{geo}^*(X),$
$gr_{\gamma}^*(X)$ related to AHss and Chern classes.
For their definition, see [Ga-Za], [Za], [Ju], [To3], [Ya3,5]),  while definition itself is not used in this paper.

Let $X$ be a topological space.  Let $K^*_{top}(X)$ be the complex $K$-theory.
Consider the  AHss
 \[ E_2^{*,*'}(X)\cong H^*(X)\otimes K^{*'}_{top}\Longrightarrow K^*_{top}(X).\]
Then $gr_{top}^*(X)$ is characterized as 
 \begin{lemma}(Atiyah \cite{At})  $gr^*_{top}(X)
\cong E_{\infty}^{*,0}(X)$.
 \end{lemma}
 
 Let $X$ be a smooth algebraic variety over a subfield  $k$ of $\bC$.
 Let $K_{alg}^*(X)$ be the algebraic $K$-theory.
This $K$-theory can be written by the motivic $K$-theory $AK^{*,*'}(X)$ (\cite{Vo1}, \cite{Vo2},
\cite{YaA}), i.e.,
$K_{alg}^i(X)\cong \oplus _{j} AK^{2j-i,j}(X).$
In particular $K_{alg}^0(X)\cong 
AK^{2*,*}(X)$.  
We recall the motivic AHss (\cite{YaA}, \cite{YaB})
 \[ AE_2^{*,*',*''}(X)\cong H^{*,*'}(X;K^{*''})\Longrightarrow AK^{*,*'}(X)\]
where $K^{2*}=AK^{2*,*}(pt.).$
 Here we have 
\[ AE_2^{2*,*,*''}(X)\cong H^{2*,*}(X;K^{*''})\cong CH^*(X)\otimes K^{*''}.\]
Hence 
$AE_{\infty}^{2*,*,0}(X)$ is a quotient of $ CH^*(X)$,
since  $d_rAE_{r}^{2*,*,*''}(X)=0$ for smooth $X$.
Then $gr_{geo}^*(X)$ is characterized as  
 \begin{lemma}  (\cite{To3}, Lemma 6.2 in \cite{YaF})  We have
\[gr^{2*}_{geo}(X)\cong AE _{\infty}^{2*,*,0}(X)\cong CH^*(X)/I(1)\]
where $I(1)=\cup_rIm(d_r)$.
 \end{lemma}
  \begin{lemma}  (\cite{To3}, \cite{YaF})  The realization  induces $gr_{geo}^*(X)\to gr_{top}^*(X(\bC))$.
\end{lemma}

At last, we consider the gamma filtration.
Let $\lambda^i(x)$ be the exterior power of the vector bundle
$x\in K_{alg}^0(X)$ and $\lambda_t(x)=\sum \lambda^i(x)t^i\in K_{alg}^*(X)[t]$.  The total 
$\gamma$-class is defined by
$\lambda_{t/(1-t)}(x)=\gamma_t(x)$ so that $\gamma^i(x)\in K^0_{alg}(X)$ if so is $x$.
The gamma filtration is defined by using this 
$\gamma^i(x)$, 
 which relates Chern classes as so does $\lambda^i(x)$. 
There is a natural map $gr_{\gamma}^*(X)\to gr_{geo}^*(X)$. 

\begin{lemma} ( \cite{At}, \cite{To3}, \cite{YaF})
The condition $g_{\gamma}^{2*}(X)\cong gr_{top}^{2*}(X)$ 
(resp. $gr_{\gamma}^{2*}(X) \cong gr_{geo}^{2*}(X)$)
 is equivalent to that $E_{\infty}^{2*,0}(X)$ (resp. $AE_{\infty}^{2*,*,0}(X)$)
is (multiplicatively) generated by 
 Chern classes in $H^{2*}(X)$ (resp. $CH^*(X)$).
 \end{lemma}

Recall that $A\tilde K(1)^{*,*'}(X)$ is  the algebraic Morava $K$-theory with $\tilde K(1)^{*}=\bZ_{(p)}[v_1,v_1^{-1}].$  
The motivic $K$-theory $AK^{*,*'}(X)$ is written by
using this Morava $K$-theory $A\tilde K(1)^{*,*'}(X)$
and the Bott periodicity $B$ with $deg(B)=(-2,-1)$,
as the topological case stated in $\S 2$ 
   \cite{YaF}
\[A\tilde K^{*,*'}(X)\cong A\tilde K(1)^{*,*'}(X)\otimes _{\tilde K(1)^*}\bZ[B,B^{-1}]\quad
with \ v_1=B^{p-1}.\]
\begin{lemma} (\cite{YaF})  
Let  $E(AK)_r$ (resp. $E(A\tilde K(1))_r$) be 
the AHss converging to $AK^{*,*'}(X)$ (resp.
$A\tilde K(1)^{*,*'}(X)$).  Then
\[gr_{geo}^{2*}(X)\cong E(AK)^{2*,*,0}_{\infty}
\cong E(A\tilde K(1))^{2*,*,0}_{\infty}.\]
\end{lemma}

From the above lemmas,  it is sufficient to consider the Morava $K$-theory
$A\tilde K(1)^{2*,*}(X) $ when we want to study $AK^{2*,*}(X)$.
For ease of notations, let us write simply (for smooth $X$)
\[K^{2*}(X)=A\tilde K(1)^{2*,*}(X),\quad so\ that \ \  K^*(pt.)\cong \bZ_{(p)}[v_1,v_1^{-1}],\]
and $k^{2*}(X)=A\tilde k(1)^{2*,*}(X)$, so that
$k^*(pt.)\cong \bZ_{(p)}[v_1].$
Hereafter of this paper, we only use  this 
Morava $K$-theory
$K^{2*}(X)$  instead of
$AK^{2*,*}(X)$ or $K_{alg}^0(X)$.

For each $n>1$ we can also define  the integral Morava $\tilde K(n)^*$-theory and 
let $gr(n)^*(X)$ be the induced graded ring, that is,
\[ gr(n)^*(X)=E(A\tilde K(n))_{\infty}^{2*,*,0}.\]
where $E(A\tilde K(n))^{*,*',*''}_r$ is the AHss
 converging to
the motivic  integral Morava K-theory $A\tilde K(n)^{*,*'}(X)$.
In particular, we see $gr(1)^*(X)\cong gr_{geo}^*(X)$.
The case $n=2$, the ring  $gr(2)^*(X)$
will used in $\S 10$ below.

We have the surjection 
$ E_2^{2*,*,0}\cong CH^*(X)\twoheadrightarrow
E_{\infty}^{2*,*,0}\cong gr(n)^*(X),$
which induces the isomorphism 
\[ CH^*(X)/I(n)\cong gr(n)^*(X), \quad I(n): ideal \ in\ CH^*(X).\]
Note that for a sufficient  large $n$, we see $I(n)=0$, but 
$res: A\tilde K(n)^*(X)\to A\tilde K(n)^*(\bar X)$ is far from
an isomorphism.
Moreover,  it seems that there is no direct relation
to $gr_{\gamma}^*(X)$, when $n>1$. 

Hereafter this paper, for ease of notations, we write
\[\tilde  k(n)^*(X)=A\tilde k(n)^*(X), \quad and \quad \tilde K(n)^*(X)=\tilde k(n)^*(X)[v_n^{-1}]\]
so that $\tilde k(n)^*\cong \bZ_{(p)}[v_n]$. 
(So $K^*(X)=\tilde K(1)^*(X)$.)

Similar but different approaches using the Morava $K(n)$-theory are given by Sechin and Semenov \cite{Sec-Se}, \cite{Sec}.

At last of this section,
we consider the restriction map
$ res_K: K^*(X)\to K^*(\bar X).$
By Panin \cite{Pa}, it is known that $K_{alg}^0(\bF)$ is torsion free
for each  twisted flag varieties $\bF=\bG/B_k$.
For each simply connected Lie group $G$, it is well known that
$res_K$ is surjective from Chevalley.  
Hence we have
\begin{thm} (Chevalley, Panin \cite{Pa})
When $G$ is simply connected, $res_K$ is 
an isomorphism.
This implies $K^*(\bF)\cong K^*(\bar \bF)\cong K^*(G/T)$.
\end{thm}
Hence we see 
$ gr_{\gamma}(\bF)\cong gr_{\gamma}(G_k/B_k)\cong 
gr_{\gamma}(G/T).$
Moreover if $\bG$ is versal, then we have
$gr_{\gamma}(\bF)\cong gr_{geo}^*(\bF)$
from Theorem 3.1 and Lemma 4.4.  Thus we have
Lemma 1.1 in the introduction.

\section{connective $K$-theory $k^*(\bar R(\bG))$}

Let $h^*(-)$ be one of  coneccted oriented 
integral theories
$\Omega^*(X), k^*(-)=\tilde k(1)^*(X),$
$\tilde k(n)^*(X)$.
.  Since $H^*(G/T)$ is
torsion free and the AHss collapses,  we have
the $h^*$-algebra isomorphism
\[h^*(\bar \bF)\cong h^*(G/T)\cong h^*[y_1,...,y_k]\otimes S(t)/(p_1,...,p_r,\bar b_1,...,\bar b_{\ell})\]
where $p_i=y_i^{p^{r_i}}$ $mod(S(t)^+)$, $\bar b_i=b_i$
$mod(I_{\infty}S(t)^+)$.  

Recall that additively 
 \[h^*(\bar R(\bG)/p\cong h^*(\bar \bF)/(p,S(t)^+)
\cong 
h^*\otimes P(y)/p.\]
 (Recall $h^*(\bT)\cong h^*(\bar \bT)$ for the Tate motive $\bT$ in Theorem 3.3, and 
it is contained in $h^*\otimes S(t)^+$.) 
Hence we can identify
\[   h^*(\bar \bF)\supset h^*\otimes P(y)
\stackrel {(*)}{\to}  h^*(\bar R(\bG)) \cong h^*\otimes P(y)\]
 where $(*)=pr|_{k^*\otimes P(y)}$ is additive isomorphism.

For the following Lemma 5.1-Corollary 5.5, 
we assume $h^*(X)=\tilde k(n)^*(X)$.
(However, note $\tilde K(n)^*(\bar \bF)\not \cong
K(n)^*(\bF)$ for $n\ge 2$, in general.)

From Lemma 4.3, we see
\[gr(n)^*(X)/p   
\cong CH^*(X)/(p,I(n))\]
\[\cong (h^*(X)\otimes _{h^*}\bZ/p)/I(n)\cong h^*(X)/(I(n),I_{\infty}).\]
Here  note $I(n)$ is  the ideal of (higher) $v_n$-torsion elements in $h^*(X)\otimes_{h^*}\bZ/p$.  Hence for $a\in h^m(X)$, 
the following conditions are equivalent
\[a =0 \in gr(n)^m(X)/p \ \ 
\Longleftrightarrow  \ \ 
a\in (I_{\infty},I(n))h^*(X).\]

\begin{lemma}
Let $a\not =0\in \tilde K(n)^0(X)$.  Then there is
$s\ge 0$  such that $v_1^{-s}a=a'\in k^*(X)$,
$a'\not =0 \in gr(n)^{2s}(X),$ and 
$a'\not \in Im(I_{\infty}h^*(X))$.
\end{lemma}
\begin{proof}
Suppose that  $a\not =0\in \tilde K(n)^0(X)$.
Since $K(n)^*(X)\cong h^*(X)[v_1^{-1}]$, there is
$s$ such that 
\[ a=v_n^{s}a'\quad for \ some \ a'\in h^{2s}(X).\]
Let us take the largest such $s$.
Then $a'$ is a $k^*$ module generator of $h^*(X)/p$.
So $s$ must be nonnegative.
\end{proof}

\begin{lemma}
For $a\in h^m(\bar R(\bG)) \subset  h^*\otimes P(y)$, if
$a=0\in gr(n)^*(R(\bG))/p$, then
\[a\in Im(I_{\infty}A(b)) \subset  h^*(\bar R(\bG))\cong
h^*\otimes P(y).\]
\end{lemma}
\begin{proof}
The fact $a=0\in gr(n)^*(R(\bG))/p$ implies 
that $a\in I_{\infty}h^*(R(\bG))$ in $h^*(\bar R(\bG))
\cong k^*\otimes P(y)$ (which is no $v_n-torsion$).
When $\bG$ is versal, we see 
\[  a=0 \quad  mod(I_{\infty}(h^*(R(\bG))
 )=mod(I_{\infty}(h^*\otimes S(t))) =
mod(I_{\infty}\otimes A(b)) \]
 in $h^m(\bar R(\bG))$.
\end{proof}

The converse of the above lemma seems difficult to
show.  We only see 
\begin{lemma}  
  If $a\not=0\in gr(n)^m(R(\bG))/p$ and $a\in
Im(I_{\infty}A(b))\subset h^*\otimes P(y)/p$, then there is $a'\in Im(I_{\infty}A(b)
\subset h^*(R(\bG))/p$
such that $a-a'=0\in gr(n)^m(X)$.
\end{lemma}

From Lemma 3.5, we show additively
\[ Ker(pr)=k^*\otimes P(y)\otimes  S(t)^+/(b)\subset h^*(\bF)\cong k^*\otimes H^*(G/T).\]

 \begin{lemma}
Let $y,y'\in h^*\otimes P(y)\subset h^*(\bF)$.
Moreover $a,a'\in Ker(pr)$ with 
$a\in I_{\infty}^s h^*(\bar \bF) $,  $a'\in I_{\infty}^r
h^*(\bar \bF)$.  Then  we have in $h^*(\bar \bF)$
\[ (y+a)(y'+a')=yy'\quad mod(Ker(pr), I_{\infty}^{s+r+1}).\]
That is, $pr((y+a)(y'+a'))=pr(yy')$ in $h^*(\bar R(\bG))/
(I_{\infty}^{s+r+1}).$
\end{lemma}
\begin{proof}
Let $x=y+a,x'=y'+a'\in h^*(\bar \bF)$ and 
 $a\in I_{\infty}^r h^*(\bar \bF)$.            
From Lemma 3.5,  we can write
\[ a=\sum a_Iy_It_I \quad a_I\in I_{\infty}^s,\ \ t_I
\in S(t)^+/(b)\quad in\ h^*(\bar \bF).\]

Similarly we take $a'=\sum a_{I'}y_{I'}t_{I'}$.  Then 
\[ x\cdot x'-y\cdot y'=x_1+x_2\quad in\ h^*(\bar \bF)\]
\[ with \ \ x_1=y\sum a_{I'}y_{I'}t_{I'}+y'\sum a_Iy_It_I,
\ \ 
and
\ \  x_2=(\sum a_Iy_It_I)(\sum a'_{I'}y_{I'}t_{I'}).\]
Here we see $pr(x_1)=0$
also from Lemma 3.5. Let us write
\[ x_2=x_{21}+x_{22} \quad with\  x_{21}=
\sum a_{I''}y_{I''}t_{I''}, \ a_{I''}\in I_{\infty}^{r+s}\]
\[and \quad x_{22}=\sum a_{J}y_Jt_Jb_J  \  with \
a_J\in I_{\infty}^{r+s},\ b_J\in A_{(p)}^+(b)= \bZ_{(p)}[b_1,...,b_{\ell}]^+. \]
Here we have 
$pr(x_{21})=0$ and 
$x_{22}=0\ mod(I_{\infty}^{r+s+1}), $
since $b_J\in (I_{\infty}).$
\end{proof}
\begin{cor}  Let $pr(b_{i_j})=pr(b_{i_j}')$ for $b_i'\in k^*(\bar \bF)$.  Then
\[ pr(b_{i_1}...b_{i_s})=pr(b_{i_1}'...b_{i_s}')\quad in\ 
k^*(\bar R(\bG))/(I_{\infty}^{s+1}).\]
\end{cor}
\begin{proof} 
Let $x=b_{i_1}...b_{i_{s-1}}$.  Then $x\in I_{\infty}^{s-1}$.
So $y=pr(x)$ and $a=x-y\in Ker(pr)$  are also in 
$(I_{\infty}^{s-1})$. Hence we have the result from the preceding lemma.
\end{proof}

The results Lemma 5.1-Corollary 5.5 hold also
for $\Omega^*(X)=k(\infty)^*(X)$.
Here we identify
$h^*(-)$ (resp. $\tilde K(n)^*(\infty),$  $gr(\infty)^*(X),$  $v_n$-torsion)
by  $\Omega^*(X)$ (resp. , 
 $\Omega ^*(X)[v_i^{-1}|i\ge 1]$,
$gr(\infty)^*(n)$,
$(v_i|i\ge 1)$).

\section{The spin group $Spin(2\ell+1)$ and $p=2$}

At first we consider the
orthogonal groups $G=SO(m)$ and $p=2$
, while it is not simply connected.
The $mod(2)$-cohomology is written as ( see for example \cite{Mi-Tod}, \cite{Ni})
\[ grH^*(SO(m);\bZ/2)\cong \Lambda(x_1,x_2,...,x_{m-1}) \]
where $|x_i|=i$, and the multiplications are given by $x_s^2=x_{2s}$.

For ease of argument,  we only consider the case
$m=2\ell+1$ so that
\[ H^*(G;\bZ/2)\cong P(y)\otimes \Lambda(x_1,x_3,...,x_{2\ell-1}) \]
\[ grP(y)/2\cong \Lambda(y_2,...,y_{2\ell}), \quad 
letting\ y_{2i}=x_{2i}\ \ (hence \ y_{4i}=y_{2i}^2).\]

The Steenrod operation is given as 
$Sq^k(x_i)= {i\choose k}(x_{i+k}).$
The $Q_i$-operations are given by Nishimoto \cite{Ni}
\[Q_nx_{2i-1}=y_{2i+2^{n+1}-2},\qquad Q_ny_{2i}=0.\]
It is well known that   the transgression 
$b_i=d_{2i}(x_{2i-1})=c_i$ is the  $i$-th elementary symmetric function
on $S(t)$. 
 Moreover we see that 
$Q_0(x_{2i-1})=y_{2i}$ in $H^*(G;\bZ/2)$.
From Corollary 2.1, we have
\begin{cor} In $BP^*(G/T)/\II$, we have  
\[c_i= 2y_{2i}+\sum_{n\ge 1} v_ny(2i+2^{n+1}-2) \]
for some $y(j)$ with $\pi^*(y(j))=y_{j}$.
\end{cor}

We have $c_i^2=0$ in $CH^*(\bF)/2$ from
the natural inclusion $SO(2\ell+1)\to Sp(2\ell+1)$
(see \cite{Pe}, \cite{YaC}) for the symplectic group
$Sp(2\ell+1)$.
Thus we have 
\begin{thm} (\cite{Pe}, \cite{YaC}) 
Let $(G,p)=(SO(2\ell+1),2)$ and $\bF=\bG/B_k$
be versal.  
Then $CH^*(\bF)$ is torsion free, and 
\[ CH^*(\bF)/2\cong S(t)/(2,c_1^2,...,c_{\ell}^2),\quad
CH^*(R(\bG))/2\cong \Lambda(c_1,...,c_{\ell}).\]
\end{thm}
\begin{cor} We have 
$ gr_{\gamma}(\bF)\cong gr_{geo}(\bF)\cong CH^*(\bF).$
\end{cor} 

From Corollary 2.2, we see
$ c_i=2y_i+v_1 y_{2i+2}$ $mod (v_1^2).$
Here we consider
the $mod(2)$ theory version
$ res_{K/2}: K^*(\bF)/2 \to  K^*(\bar \bF)/2.$
\begin{lemma}
We have  $ Im(res_{K/2})\cong K^*\otimes \Lambda(y_{2i}|2i\ge 4).$  Hence $res_K$ is not surjective.
\end{lemma}

Now we consider  
$Spin(2\ell+1)$.
Throughout this section, let $p=2$,  $G=SO(2\ell+1)$
and $G'=Spin(2\ell+1)$.
It is well known that
$ G/T\cong G'/T'$ for the maximal torus $T'$ of the spin group.

  By definition, we have the 
$2$ covering $\pi:G'\to G$.
It is well known that 
$\pi^*:   H^*(G/T)\cong H^*(G'/T')$.
Let $2^t\le \ell < 2^{t+1}$, i.e. $t=[log_2\ell]$.
The mod $2$ cohomology is
\[ H^*(G';\bZ/2)\cong H^*(G;\bZ/2)/(x_1,y_2)\otimes
\Lambda(z) \]
\[ \cong P(y)'\otimes \Lambda(x_3,x_5,...,x_{2\ell-1})\otimes \Lambda(z),\quad |z|=2^{t+2}-1\]
where
$P(y)\cong \bZ/2[y_2]/(y_2^{2^{t+1}})\otimes P(y)'$.  
(Here $z$ is defined by $d_{2^{t+2}}(z)=y^{2^{t+1}}$ for $0\not =y\in H^2(B\bZ/2;\bZ/2)$ in the spectral sequence induced from
the fibering $G'\to G\to B\bZ/2$.)
Hence
\[ grP(y)'\cong \otimes _{2i\not =2^j}\Lambda(y_{2i})\cong
\Lambda(y_6,y_{10},y_{12},...,y_{2\bar \ell})\]
where $\bar \ell=\ell-1$ if $\ell=2^j$ for some $j$, and
$\bar \ell=\ell$ otherwise.

The $Q_i$ operation for $z$ is given by Nishimoto 
\cite{Ni}
\[ Q_0(z)=\sum _{i+j=2^{t+1},i<j}y_{2i}y_{2j}, \quad 
 Q_n(z)=\sum _{i+j=2^{t+1}+2^{n+1}-2,i<j}y_{2i}y_{2j}\ \ for\ n\ge 1.\]

We know that 
\[ grH^*(G'/T')/2\cong P(y)'\otimes S(t')/(2,c_2',.....,c_{\ell}',c_1^{2^{t+1}}).\]
Here $c_i'=\pi^*(c_i)$ and $d_{2^{t+2}}(z)=c_1^{2^{t+1}}$ in the spectral sequence
converging $H^*(G'/T')$.

Take $k$ such that $\bG$ is a versal $G_k$-torsor so that
$\bG'_k$ is also a versal $G_k'$-torsor.  Let us write
$\bF=\bG/B_k$ and $\bF'=\bG'/B_k'$.  Then
\[ CH^*(\bar R(\bG'))/2\cong P(y)'/2,\quad and \quad
    CH^*(\bar R(\bG))/2\cong P(y)/2.\]

The Chow ring $CH^*(R(\bG'))/2$ is not computed yet
(for general $\ell$),
while we have the following lemmas.
\begin{lemma}
Let $2^t\le \ell<2^{t+1}$and $\bG$ be versal. 
Then there  is a surjection
\[ \Lambda(c_2',...c_{\bar \ell}')\otimes \bZ/2[c_1^{2^{t+1}}]
\twoheadrightarrow CH^*(R(\bG')) /2.\]
\end{lemma}

Let $G''=Spin(2\ell+2)$.  Then the inclusion
$i:G'\subset G''$ induces 
\[ CH^*(\bar R(\bG''))\cong P(y)'\stackrel{i^*}{\cong}
     CH^*(\bar R(\bG')).\]
Hence by a (main) theorem by Vishik-Zainoulline
\cite{Vi-Za}, we have

\begin{lemma}  Let $G'=Spin(\ell+1)$ and $G''=Spin(2\ell+2)$.  Then
$R(\bG')\cong R(\bG'')$.
\end{lemma}
 
Therefore, in this paper, we consider only $G=Spin(2\ell+1)$,  but the arguments for
$G=Spin(2\ell+2)$ are similar.

\begin{lemma}  The restriction $res_K:
K^*(\bF')\to K^*(\bar \bF') $ 
 is isomorphic.
In fact, we have $K^*(R(\bG'))/2\cong 
K^*/2\otimes  \Lambda(c_i'|i\not =2^j-1).$
\end{lemma}



\section{$Spin(7),Spin(9)$}

From  this section to the last section,
we only consider $G=Spin(2\ell+1)$.
Hereafter in this paper, we change notations as follows.
 Let us write the element $c_1^j$, in the preceding section, by $e_j$,
and take off the dash of $c_i'$, i.e.
\[ e_{2^{t+1}}=c_1^{2^{t+1}},\quad c_2=c_2',\ c_3=c_3',\ ....,\ c_{\ell}=c_{\ell}'.\]

Now we consider examples. 
For groups $Spin(7),$ $Spin(9)$, the rank $\ell=3,4$
respectively and the cohomology is given
\[ grH^*(G;\bZ/2)\cong 
\begin{cases}
\bZ/2[y_6]/(y_6^2)\otimes \Lambda(x_3,x_5,z_7)\quad G=Spin(7)\\
\bZ/2[y_6]/(y_6^2)\otimes \Lambda(x_3,x_5,x_7,z_{15})
\quad G=Spin(9).\end{cases}\]

We can take  $v_1y_6=c_2, 2y_6=c_3$ in $k^*(\bar \bF)$.  The Chow groups of the Rost motives $R(\bG))$ are isomorphic
\[ CH^*(R(\bG))\cong gr_{\gamma}^*(R(\bG))\cong
\bZ_{(2)}\{1,c_3\}\oplus \bZ/2\{c_2\}.\]
 The Chow ring 
of the flag manifolds are \cite{YaC}
\[gr_{\gamma}(G/T)\cong  CH^*(\bF)\cong \begin{cases}
      S(t)/(2c_2,c_2^2,c_3^2,c_2c_3,e_4)\quad G=Spin(7),\ell=3\\
  S(t)/(2c_2,c_2^2,c_3^2,c_2c_3,e_8,c_4)\quad G=Spin(9),
\ell=4 
\end{cases} \]
where $S(t)=\bZ_{(2)}[t_1,...,t_{\ell}].$

\section{$Spin(11)$}

In this section we consider $G=Spin(11)$.
The cohomology is written as 
\[H^*(G;\bZ/2)\cong \bZ/2[y_6,y_{10}]/(y_6^2,y_{10}^2)\otimes \Lambda(x_3,x_5,x_7,x_9,z_{15}).\]
By Nishimoto, we know $Q_0(z_{15})=y_6y_{10}$. It implies
$2y_6y_{10}=d_{16}(z_{15})=e_8$.
Since $y_{top}=y_6y_{10}$, we have $t(G)=2$.
From Corollary 2.1, we can take in $k^*(\bar \bF)$
such that $c_2=v_1y_6 ,c_4=v_1y_{10}$.  Moreover
in $k^*(\bar \bF)/(I_{\infty}^2)$ we have
$c_3=2y_6,c_5=2y_{10}$.

\begin{lemma}
For $x\in k^*\otimes S(t)$, let $x\in I_{\infty}^3
 k^*(\bar \bF).$
Then  $x\in I_{\infty}\cdot S(t)$, i.e., $x=0$ in
 $gr_{\gamma}(\bF)/2$.
\end{lemma}
\begin{proof}  
Take off product of  $c_i$ and $t\not =0\in S(t)/(b)$
from $x$, we get the result.
For example, let  $x=8y_6y_{10}\ mod (4v_1).$
 We consider
$x'=x-2c_3c_5$ or $x'=x-4e_8$.  Then $x'\in I_{\infty}^3$ and 
$|x'|>|x|$.  Since $x=0\in CH^*(\bF)/2$ for $|x|>2dim(\bF)$,
we get the result, by continuing this argument.
\end{proof}

\begin{thm} (\cite{Ka2}, \cite{YaC}, \cite{YaG})
 For $(G,p)=(Spin(11),2)$, we have the isomorphisms
\[gr_{\gamma}(R(\bG))/2\cong \bZ/2\{1, c_2,c_3,c_4,c_5,c_2c_4, e_8\},\]
where $c_2c_4=pr(c_2\cdot c_4)$ for the cup product $\cdot$ in $gr_{\gamma}^*(\bF)$, and 
\[ gr_{\gamma}(\bF)/2\cong S(t)/(2,c_ic_j, c_ie_8,e_8^2|2\le i\le j\le 5,\ (i,j)\not
=(2,4)).\]
\end{thm}
\begin{proof}
We will compute $gr_{geo}^*(R(\bG))/2\cong
gr_{\gamma}^*(R(\bG))/2$.
Consider the restriction map $res_K$
\[ K^*(R(\bG))\cong  
 K^*\{1,c_2,c_4,c_2c_4\}
 \stackrel{\cong }{\to} K^*\{1,y_6,y_{10},y_6y_{10}\}
\cong K^*(\bar R(\bG))\]
by $res_K(c_2)= v_1y_6$, $res_K(c_4)=v_1y_{10}$
in $K^*(\bar \bF)$.
Note $c_2c_4\not =0 \in CH^*(R(\bG))/2$,
from the fact that $v_1y_6y_{10}\not \in Im(res_K)$.

Next using $res_K(c_3)=2y_6$, $res_K(c_5)=2y_{10}$,
and $res_K(e_8)=2y_6y_{10}$,  we have
\[grK^*(R(\bG))\cong
K^*/2\{c_2,c_4,c_2c_4\}\oplus
   K^*\{1,c_3,c_5,e_8\}.\]
From this, we show the first isomorphism.
Note generators $c_2,...,e_8$ are all nonzero
in $CH^*(R(\bG))/2$ by Corollary 3.7.
(Of course $e_8\not =0$ and $e_8^2=0$ in 
$CH^*(R(\bG))/2$, since
$y_6y_{10}\not \in res_{CH}$ and $|e_8^2|>|y_6y_{10}|$)

Let us write $y_{6}^2=y_6t+y_{10}t'\in CH^*(\bar \bF)/2$ for $t\in S(t)$.  Then
\[ c_2c_3=2v_1y_6^2=2v_1(y_6t+y_{10}t)'=v_1(c_3t+c_5t')\quad in \ K^*(\bar \bF)\cong K^(\bF).\]
Hence we see $c_2c_3=0$ in $CH^*(\bF)/I(1)\cong 
gr_{\gamma}^*(\bar \bF)$.  
We also note
\[c_2c_5=c_3c_4=v_1e_8,\quad  c_3c_5=4y_6y_{10}=2e_8 \quad in\ K^*(\bar \bF)\ \ (so\ in\ k^*(\bar \bF)).\]

By  arguments for $gr_{geo}^*(\bF)$, similar to Lemma 3.4, we can show the second isomorphism.  For example,
\[ x=c_2e_8=v_1y_6\cdot 2y_6y_{10}=v_1(2y_6^2y_{10})
=v_1(2y_6t+2y_{10}t').\]
Then we consider $x'=x-v_1(c_3t+c_5t')$.  Then $|x'|>|x|$, and continue this argument, we get the result.

The other cases are proved similarly.
\end{proof}

  Karpenko computes $CH^*(\bF)$ and proves (\cite{Ka2}) the following theorem.
\begin{thm} (\cite{Ka2})
Let $G=Spin(11)$ and $\bG$ be versal.
Then $gr_{\gamma}(R(\bG))\cong CH^*(R(\bG))$.
\end{thm}
\begin{proof}
(Karpenko \cite{Ka2})
We consider the following commutative diagram
\[\begin{CD}
CH^*(\bar R(\bG)) @>{N}>>
CH^*(R(\bG)) @>>> CH^*(R(\bG))/N\\
@V{\cong}VV @V{f_1}VV  @V{f_2}VV \\
gr_{geo}^*(\bar R(\bG))
@>{N}>>
gr_{geo}^*( R(\bG))
@>>>
gr_{geo}^*(R(\bG))/N
\end{CD}
\]
Since  $CH^*(\bar R(\bG))$ is torsion free,
the both norm maps $N$ are injective.
Hence if $f_2$ is isomorphic, then so is $f_1$.

The norm map $N$ is given by  (using $res\cdot N \cdot res=2res$ and $CH^*(\bar R(\bG))$ is torsion free)
\[ N(y_6)=c_3,\quad N(y_{10})=c_5,\quad N(y_6y_{10})=e_8.\]
Hence $CH^*(R(\bG))/N$ (and $gr_{geo}^*(R(\bG))/N$)
is isomorphic to a quotient of
\[ \Lambda_{\bZ}(c_2,...,c_5,e_8)/Ideal(2,c_3,c_5,e_8)
\cong \Lambda(c_2,c_4).\]
We see from the preceding theorem that
$ gr_{ geo}(R(\bG))/N\cong
  \Lambda(c_2,c_4).$
So $f_1$ is isomorphic.
\end{proof}
Hence the ring in Theorem 7.1  
is isomorphic to  $CH^*(\bF)/2$.
Moreover,  Karpenko conjectured 
\begin{conj} (\cite{Ka1}, \cite{Ka2})  Let $G$ be simply connected and  $\bG$ be  a versal $G_k$-torsor.  Then 
$CH^*(\bF)\cong gr_{\gamma}^*(\bF)$.
\end{conj}

We recall here the algebraic cobordism theory
$\Omega^*(X)$ defined in $\S 4$.
For a nonzero pure symbol $a$ in the Milnor $K$-theory $K_{n+1}^M(k)/2$, let $V_a$ be the norm variety.  
The usual Rost motive $R_n$ is defined by as an irreducible motive $M(V_a)$.

Then we know (\cite{Vi-Ya}, \cite{YaB}) that the restriction
\[ res_{\Omega}:\Omega^*(R_n)\to \Omega^*(\bar R_n)
\cong BP^*[y]/(y^p)\quad |y|=2(p^n-1)/(p-1) \]
is an injection.  Moreover, we have  
\[Im(res_{\Omega})=BP^*\{1\}\oplus I_n\otimes \bZ_{(p)}[y]^+/(y^p)\]
where $I_n=(p,...,v_{n-1})$ and  $\bZ_{(p)}[y]^+/(y^p)=
\bZ_{(p)}\{y,...,y^{p-1}\}.$

When $G=Spin(7),Spin(9)$, the motive $R(\bG))\cong R_2$
and $Im(res_{\Omega})\cong BP^*\{1\}
\oplus I_2\{y_6\}$.
The similar fact holds for $R(\bG)$ with $G=Spin(11)$.
\begin{cor}  Let $G=Spin(11)$ and $\bG$ be versal. 
Then the restriction map
\[ res_{\Omega}:\Omega^*(R(\bG))\to 
\Omega^*(\bar R(\bG))
\cong BP^*\{1,y_6,y_{10},y_6y_{10}\}\]
is an injection, and 
\[Im(res_{\Omega})=BP^*\{1\}\oplus (2,v_1)\{y_6,y_{10}\}
\oplus (2,v_1^2)\{y_6y_{10}\}.
\]
\end{cor}
\begin{proof}
Since $CH^*(X)\otimes BP^*\cong gr\Omega^*(X)$, we have
\[ gr\Omega ^*(R(\bG))\cong BP^*/2\{c_2,c_4,c_2c_4\}\oplus 
BP^*\{1,c_3,c_5,e_8\}.\]
This induces an injection to
\[ gr(Im_{\Omega})\cong 
BP^*/2\{v_1y_6,v_1y_{10}, v_1^2y_{6}y_{10}\}\oplus
  BP^*\{1,2y_6,2y_{10},2y_{6}y_{10} \}  \]
\[ \subset gr\Omega^*(\bar R(\bG))= \Omega^*(\bar R(\bG))/2\oplus 2\Omega^*(\bar R(\bG)).\]
\end{proof}

We give here a conjecture
\begin{conj} Let $G$ be a connected compact Lie group 
and $\bG$ be a $G_k$-torsor. Then the restriction map \[res_{\Omega}: \Omega^*(R(\bG))\to \Omega^*(\bar
R(\bG))\subset  BP^*\otimes P(y)\]
is injective, i.e., $\Omega^*(R(\bG))$ is torsion free.
\end{conj}
When $G$ is simply connected and $\bG$ is versal, this conjecture is weaker than the Karpenko conjecture.
\begin{lemma} 
If $K^*(X)\cong K^*(\bar X)$
and  $gr_{geo}^*(X)\cong  CH^*(X)$, then 
the restriction map $res_{\Omega}$ is injective.
\end{lemma}
\begin{proof}
Let $0\not =x\in Ker(res_{\Omega})$.
  In the  AHss $E_r^{2*,*,*'}$
 converging to
$\Omega^*(X)$, let 
\[ x=ax'\quad a\in BP^*,\ 0\not = x'\in CH^*(X).\]

Suppose $a\in BP^0\cong \bZ_{(p)}$, that is,
$0\not=x\in CH^*(X)\cong gr_{geo}^*(X)$.  Hence  
$x$ is non-zero in $K^*(X)\cong K^*(\bar X)$.
Since $K^*(\bar X)\cong \Omega^*(\bar X)
\otimes _{BP^*}K^*$, we see $0\not =x\in \Omega^*(\bar X)$, and this is a contradiction.

Suppose $a\in BP^{<0}$.
Considering the Quillen operation
 $r_{p^1\Delta_{n-1}}(v_n)=v_1$ (\cite{Ra}, \cite{Ha},
\cite{YaC}).  
we can take the operation $r_{\alpha}$ such that
\[ r_{\alpha}(ax')=r_{\alpha}(a)\cdot x'=
\lambda v_1^sx'\quad for\ \ s,\ \lambda\in \bZ.\]
Hence there is a nonzero element  $\lambda'x'\not =0 \in CH^*(X)$, for $\lambda'\in \bZ$, 
but it is $v_1^s$-torsion in $E_{\infty}^{2*,*,*'}\cong
gr \Omega^*(X)$.  
Since 
$\Omega^*(X)\otimes _{BP^*}\bZ_{(p)}[v_1,v_1^{-1}]
\cong K^*(X),$
we see $0\not =\lambda'x' \in I(1). $ 
\end{proof}

{\bf Remark.}
The converse of the above lemma is not correct. In fact, for the original 
Rost motive $R_n$, $gr_{\gamma}(R_n)\not \cong
CH^*(R_n)$ for $n\ge 3$, while $res_{\Omega}$ is injective.


\section{$Spin(13)$}

In this section,  we consider the case $Spin(13)$, i.e.,
$\ell=6$.  Then we have
\[ grP(y)\cong \Lambda(y_6,y_{10},y_{12}).\]
Hence $y_{top}=y_6y_{10}y_{12}$.
Since  
$2^t\le \ell<2^{t+1}$, so $t=2$.
Hence 
$e_8$ exists in $CH^*(R(\bG))$.  Note $e_8^2=pr(e_8\cdot e_8)=0$ in $CH^*(R(\bG))$
since $|e_8^2|>|y_{top}|$.
We also note $t(G)=2^2=4$ by Totaro, 
in fact $e_8c_6=4y_{top}$.
 
{\bf Note.}   From Corollary 5.5, the projection $pr(c_{i_1}...c_{i_s})$ is 
uniquely determined in $ k^*(\bar R(\bG))/(I_{\infty}^{s+1})\cong k^*\otimes P(y)/(I_{\infty}^{s+1})$.

We take $y_6,y_{10},y_{12}$ such that
\[ c_2=v_1y_6,\ \ c_4=v_1y_{10}, \ \ y_{12}=y_6^2\quad in\
k^*(\bar \bF).\]
(For example, let $ y_6=v_1^{-1}c_2$ and let us write 
$y_{12}=y_6^2$ in $k^*(\bF)$.)
Recall that  the invariant ideal $I_{\infty}$
in $k^*$-theory is 
\[I_{\infty}=Ideal(2,v_1)\subset k^*(\bar R(\bG))
\cong \bZ_{(2)}[v_1]\otimes P(y). \]
Note that $pr(I_{\infty})\subset I_{\infty}$.
Then we can take $c_i\in k^*(R(\bG))$ such that
in $k^*(\bar R(\bG))/(I_{\infty}^3)$,
\[c_3=2y_6+\mu v_1^2y_{10}, \ \ 
 c_5=2y_{10}+v_1y_{12},\ \ 
 c_{6}=2y_{12}+\mu 'v_1^2y_{6}y_{10},\quad \mu,\mu'\in
\bZ_{(2)},\]
\[ e_8=2y_6y_{10}+v_1y_6y_{12}.\]
Here we used $P(y)^*=0$ for $*=8,14,20$,
and $ Q_1x_9=y_{12},\ Q_1z_{15}=y_6y_{12}$.
Moreover, we write (new) $c_3$ (resp. $c_6$) as the (old) element $c_3-\mu v_1c_4$ (resp. $c_6-2\mu'c_2c_4$)
such that  ($mod(I_{\infty}^3)$)
\[ c_3=2y_{6},\quad c_6=2y_{12}.\]

Let us write $\bZ_{(p)}$-free module
\[ \Lambda_{\bZ}(a_1,...,a_n)=\bZ_{(p)}\{a_{i_1}...a_{i_s}|
1\le i_1<...<i_s\le n\}.\]
\begin{lemma}
Let $A(c)=\Lambda_{\bZ}(c_2,...,c_6,e_8).$
For $x\in A(c)$, let $x\in I_{\infty}^4
 k^*(\bar R(\bG)).$
Then  $x\in I_{\infty}\cdot A(c)$, i.e., $x=0$ in
 $gr_{\gamma}(R(\bG))/2$.
\end{lemma}
\begin{proof}  
Let $x\in I_{\infty}^4k^*(\bar R(\bG))$.  Then $x$ is a sum of
\[ v_1^4y,\ 2v_1^3y,\ 4v_1^2y,\ 8v_1y,\ 16y\quad for \ y\in P(y).\]

Since the torsion index $t(G)=4$,
we see $4y_6y_{10}y_{12}\in k_*(R(\bG))$.
When $y=y_iy_j$, we consider $4y-c_{i'}c_{j'}$ instead of $4y$ where $2y_i=c_{i'}, 2y_{j}=c_{j'}\ mod(v_1)$.
Then  $4v_1^2y,8v_1y,16y$ are in $I_{\infty}A(c)$.
Hence we only need to consider  the cases 
\[x=v_1^4y\quad  or \quad x=2v_1^3y.\]

Let  $x=v_1^4y$.  When  $y=y_{top}$, we consider 
\[  x'= x-c_2c_4(c_2)^2=v_1^4y_{top}-v_1y_6v_1y_{10}(v_1^2y_{12})=0\quad in\ \ k^*(\bar \bF).\]
Here $(c_2^2)=0$ in $gr_{\gamma}(\bar \bF)/2$, and hence
$x=x'=0\in gr_{\gamma}^*(R(\bG))/2$.
When $y=y_{10}y_{12}$, taking $x'=x-v_1c_4c_2^2$,
similar arguments work.
When  $y=y_6y_{10}$, we consider
\[ x'=x-v_1^2c_2c_4=0\quad in\  k^*(\bar R(\bG)).\]
This means $x=0\in gr_{\gamma}(R(\bG))/2$.
Thus we can see the lemma for $x=v_1^4y$.

Let $x=2v_1^3y$ with $y=y_{top}$.
Then we see
\[ x'=x-v_1c_2c_4c_{6}=2v_1^3y_{top}-v_1(v_1y_6)(v_1y_{10})(2y_{12})=0
\quad mod(I_{\infty}^5)\]
which is zero in $gr_{\gamma}^*(R(\bG))/2$ since
$x'=v_1^5y'$ or $x'=2v_1^4y'$.
The other cases are proved similarly.
\end{proof}  

\begin{lemma} $c_2c_3=0\in gr_{\gamma}(R(\bG))/2$.
\end{lemma}
\begin{proof}
We have
$c_2c_3=v_1y_6(2y_6)$  $mod(I_{\infty}^4)$ in $k^*(\bar R(\bG))$.
(Here we used Lemma 5.4 and that $c_3=2y_6$ in $k^*(\bar \bF)/I_{\infty}^2$ from Corollary 2.1.)
We consider 
$ x=c_2c_3-v_1c_6$ with $mod(I_{\infty}^4)$.  That is,   
\[ x=
v_1y_6(2y_6)-v_1(2y_{12})=0\quad mod(I_{\infty}^4).\]
Hence from the preceding lemma, we see $c_2c_3=0\in gr_{\gamma}(R(\bG))/2$.
\end{proof}


We will compute $gr_{geo}^*(R(\bG))/2\cong gr_{\gamma}^*(R(\bG))/2$ (by using Lemma 5.2).  We will seek
generators of 
 \[gr_2(K^*(R(\bG))=K^*(R(\bG))/2\oplus
2K^*(R(\bG))\] which are nonzero in $CH^*(R(\bG))/2$.  
If $x$ is a generator of $K^*(R(\bG))/2$ but 
is zero in $CH^*(R(\bG))/2$,  then we seek an element
$x'\not =0\in CH^*(R(\bG))/2$ with $x=v_1^sx'$ in 
$K^*(\bar R(\bG))$ for some $s\ge 1$.

We note  that
\[K^*(R(\bG))/2\cong K^*\otimes \Lambda(c_2,c_4,c_5)\]
\[\qquad \cong K^*\otimes \Lambda(y_6,y_{10},y_{12})\cong
K^*(\bar R(G))/2\]
by $c_2\mapsto v_1y_6$, $c_4\mapsto v_1y_{10}$
and $c_5\mapsto 2v_1^{-1}c_4+v_1y_{12}$.

At first we study elements in $\Lambda(c_2.c_4,c_5)$.
Of course $c_2,c_4,c_5$ are nonzero in $gr_{\gamma}^*(R(\bG))/2$.
We also know $c_2c_4\not =0\in gr_{\gamma}^*(R(\bG))/2$
since so for $G=Spin(11)$.

We have (with $mod(I_{\infty}^4)$)
\[ c_2c_5=v_1y_6(2y_{10}+v_1y_{12})=v_1e_8.\]
Hence we get
$ c_2c_5c_4=v_1e_8c_4,$
where $e_8c_4=v_1^2y_{top}\ mod(I_{\infty}^3)$ which is nonzero
in $gr_{\gamma}^*(R(\bG))/2$.
(Here we note $c_2c_4\mapsto v_1^2y_6y_{10},\ c_2c_5-c_3c_4 \mapsto v_1^2y_6y_{12}$, and they are nonzero
in $gr_{\gamma}^*(R(\bG))/2$.)

For the element $c_4c_5$,  we compute
\[c_4c_5=v_1y_{10}(2y_{10}+v_1y_{12})=v_1^2y_{10}y_{12}\ \ 
mod(I_{\infty}^3)\]
which is also non zero in  $ gr_{\gamma}^*(R(\bG))/2$.
Thus we have 

\begin{lemma}  We have some filtration and the graded ring
\[ grP(c)/(2)\cong \bZ/2\{1,c_2,c_4,c_5,c_2c_4,c_4c_5\}\oplus
\bZ/2\{ e_8, e_8c_4\} \]
\[ \cong P(c)/(2,c_2c_5,c_2c_5c_4)\oplus
\bZ/2\{ e_8, e_8c_4\}\subset gr_{\gamma}(R(\bG))/2,\]
where $P(c)=\Lambda_{\bZ}(c_2,c_4,c_5).$
\end{lemma}


We  consider   elements in $2K^*(\bar R(\bG))\cap Im(res_{k})$.
First note   
\[2y_6=c_3,\quad 2y_{12}=c_6\quad with \ mod(I_{\infty}^3).\]
and $c_3,c_6\not =0$ in $gr_{\gamma}^*(R(\bG))/2$.
Next we see (with $mod(I_{\infty}^3)$)
\[ 2v_1y_6y_{10}=c_3c_4,\quad 2v_1y_6y_{12}=c_2c_6,\quad 2v_1y_{10}y_{12}=c_4c_6.\]
Here $c_3c_4,c_2c_6,c_4c_6\not =0\in CH^*(R(\bG))/2$
since $2y_6y_{10},2y_6y_{12},2y_{10}y_{12}\not \in Im(res_{CH})$.
These $c_3c_4,c_2y_6,c_4c_6$ are $2$-torsion in $gr_{\gamma}(R(\bG))$, in fact, we  see
\[ 2c_3c_4=v_1c_2c_6,\quad  2c_2c_6=v_1c_3c_6,\quad 2c_4c_6=v_1c_5c_6.\]

Next we consider elements $4y$ in $y\in P(y)$.
In particular, 
\[ 4y_{10}=2c_5+v_1c_6,\quad
4y_6y_{10}=2e_8+c_2c_6,\quad\]
\[ 4y_6y_{12}=c_3c_6,\quad 4y_{10}y_{12}=c_5c_6,\]
here $c_3c_6,c_5c_6\not =0$ in $CH^*(R(\bG))/2$. 
Moreover
\  $4y_6y_{10}y_{12}=e_8c_6$  \ 
here $e_8c_6\not =0\in CH^*(R(\bG))/2$ since $t(G)=4$.
Thus we have
\begin{lemma}  The graded ring $gr_{\gamma}(R(\bG))$ contains
\[ \bZ/2\{c_4c_3,c_2c_6,c_4c_6\}\oplus  \bZ_{(2)}
\{c_3,c_5,c_6,e_8, c_3c_6, c_5c_6.e_8c_6\}.\]
\end{lemma}

From preceding two lemmas, we have

\begin{thm} Let $G=Spin(13)$
and $\bG$ be versal.  Then we have 
\[ gr_{\gamma}(R(\bG))
\cong P(b)/(c_2c_5,c_2c_4c_5)\oplus
\bZ/2\{c_4c_3,e_8c_4,c_2c_6,c_4c_6\}\]
\[\oplus \bZ_{(2)}\{1,c_3,c_6,e_8. c_3c_6,c_5c_6, e_8c_6\}\]
where $P(b)=\Lambda_{\bZ}(c_2,c_4,c_5)/(2c_2,2c_4)$
\end{thm}
\begin{cor} We have the additive isomorphism
\[ gr_{\gamma}(R(\bG))\cong 
    A\otimes (\bZ_{(2)}\{1,c_6\}\oplus \bZ/2\{c_4\})\oplus
\bZ/2\{c_6c_4\}\]
where $A= \bZ_{(2)}\{1,c_3,c_5,e_8\}\oplus \bZ/2\{c_2\}.$
\end{cor}
\begin{cor} Let $G=Spin(13)$, $\bG$ be versal,
and $\bF=\bG/B_k$.  To state the formula simply,
let us write $e_8=c_8$.  Then
we have additive (graded) isomorphism
\[ gr_{\gamma}(\bF)/2\cong
S(t)/(2,c_ic_j,c_mc_nc_k|(i,j)\not=(s,4), (s',6)\
for \ s\not =4,\ s'\not =6)\]
where $i,j,k,m,n,s,s'$ are one of $2,...,6,8$.
\end{cor}

Note that the above isomorphism is not that of rings.
In fact, we see 
\[c_2e_8-c_4c_6=0 \quad in\   gr_{\gamma}^*(\bF)/2.\]
We also can not get $CH^*(R(\bG))$ here, but we do some arguments using norm map.
The norm map $N$ is given  explicitly
\[ N(y_6)=2c_3,\ N(y_{10})=2c_5,\ N(y_6y_{10})=2e_8,
 N(y_{12})=2c_6,\]
\[ N(y_6y_{12})=c_3c_6,\ N(y_{10}y_{12})=c_5c_6,\ N(y_{top})=e_8c_6.\]
\begin{lemma}  We have  \ \ 
$gr_{geo}^*(R(\bG))/N\cong gr_{geo}^*(R(\bG))/
(2,c_3c_6,c_5c_6,e_8c_6).$\end{lemma} 

 
\begin{lemma}  If we can see that 
$c_2c_3=0$ and  $c_2e_8-c_4c_6=0$
in $CH^*(R(\bG))/2,$
then $gr_{\gamma}^*(R(\bG))
\cong CH^*(R(\bG))$.
\end{lemma}
\begin{proof}
By the Wu formula,  we have
\[c_2c_3=0\stackrel{P^2}{\Longrightarrow}
c_2c_5=0\stackrel{P^1}{\Longrightarrow} c_3c_5=0,\]
\[  c_2e_8-c_4c_6=0\stackrel{P^1}{\Longrightarrow} c_3e_8-c_5c_6=0\stackrel{P^2}{\Longrightarrow} c_5e_8=0.\]
Let us write $c_8=e_8$ and $c_7=0$, and 
$  A(c)'=\bZ/2\{1,c_k, c_ic_j|1\le i,j,k\le 8, i<j\}.$
Then from Theorem 9.5, we can see
\[ gr_{\gamma }^*(R(\bG))/2\cong
A(c)'/(c_2c_3,c_2c_5, c_3c_5,\ c_2e_8,c_3e_8,c_5e_8)
.\]
Hence we have the lemma.
 \end{proof}

\begin{lemma}  Let $G=Spin(13)$ and $\bG$ be versal. 
Then for the restriction map
$ res_{k}:k^*(R(\bG))\to k^*(\bar R(\bG))$, the image 
$Im(res_{k})\subset grk^*(\bar R(\bG))$ is given by
\[res(k^*(R(\bG'))
\oplus (2)\{y_{12}\}\oplus (4,2v_1, v_1^2)\{y_6y_{12},y_{10}y_{12}\}
\oplus (4,v_1^2)
\{y_{top}\}
\]
where $G'=Spin(11)$ and $(a,...,b)\subset k^*$
is the ideal generated by $a,...,b$.
\end{lemma}
\begin{proof}
%
The image of the restriction map is written (with $mod(I_{\infty}^3)) $  as follows
\[c_6\mapsto 2y_{12},\ \ c_6c_3\mapsto 4y_6y_{12},\ \ 
c_6c_5\mapsto 4y_{10}y_{12},\ \ c_6e_8\mapsto 4y_{top} \]
\[  c_6c_2\mapsto 2v_1y_6y_{12},\ \ 
c_6c_4\mapsto 2v_1y_{10}y_{12},\ \  (e_8c_5\mapsto 4v_1y_{top})\]
\[(c_3c_4-v_1e_8)\mapsto v_1^2y_6y_{12},\ \ c_4c_5\mapsto v_1^2y_{10}y_{12},\ \ 
e_8c_4\mapsto v_1^2y_{top}.\]
These facts imply the lemma from Theorem 9.5.
\end{proof}
If  the Karpenko conjecture holds for this case,
then from the above lemma, we see that the maps $res_{k}$ and $res_{\Omega}$ 
are injective (as  stated in Lemma 8.7).

\section{$Spin(15), Spin(17)$ and the counterexample}

At first, we consider the case $Spin(15)$, i.e.,
$\ell=7$.  Then we have
\[ grP(y)\cong \Lambda(y_6,y_{10},y_{12},y_{14}), \quad y_6^2=y_{12}.\]
So $y_{top}=y_6y_{10}y_{12}y_{14}$ and
there is 
$e_8$ in $CH^*(R(\bG))$.
By Totaro, it is known that the torsion index is $2^3=8$.

\begin{lemma} 
We have  $c_2c_3\not =0$ in $gr_{\gamma}(R(\bG))/2$.
\end{lemma}
\begin{proof}
The element $x=c_2c_3-v_1c_6$ is written
for $y',y''\in k^*\otimes P(y)$
\[v_1y_6(2y_6-v_1^2y')-v_1(2y_{12}+v_1y_{14}
+v_1^2y'')=-v_1^2y_{14}\quad mod(I_{\infty}^3).\]
Let us write $A(c)= \Lambda_{\bZ}(c_2,...,c_{7})\otimes \bZ[e_8]$.  
If $x=v_1x'$ with  $x'\in k^*\otimes\Lambda(c)$,
 then $x'=v_1y_{14}$.
But such $x'$ does not exist, because there does not 
exist $x_i$ such that $Q_1x_i=y_{14}$ but $Q_0x_i=0$ (note
$Q_0x_{11}=y_{12}$).
\end{proof}

Next we consider the case 
 $G=Spin(17)$ and $\bG$ be versal.  
Hence $P(y)$ is the same as the case $G=Spin(15)$, that is, $ grP(y)\cong \Lambda(y_6,y_{10},y_{12},y_{14})$.
So $y_{top}=y_6y_{10}y_{12}y_{14}$.  But by Totaro,
it is known that $t(G)=2^4=16$. Note that we have 
the element
$e_{16}\in CH^*(R(\bG))$.


\begin{lemma}  Let $G=Spin(17)$ and $\bG$ be versal.
We have $c_2c_3c_6c_7=0$ in $gr_{\gamma}(R(\bG))/2$.
\end{lemma}
\begin{proof}
The element $ x=c_2c_3c_6c_7$ is computed in $k^*(\bar \bF)$ (See Lemma 5.3.)
\[x=v_1y_6(2y_6+v_1^2y')(2y_{12}+v_1y_{14}+v_1^2y'')(2y_{14}+v_1^2y''')=0\quad mod(I_{\infty}^5)\]
for $y',y'',y''' \in k^*\otimes P(y)$, because $y_{12}^2,y_{14}^2\in I_{\infty}$.

The degrees are given  $|x|=|y_{top}|-6$ and $|v_1^3|=-6$. So
\[ (*)\quad  x=2^2\lambda v_1^3y_{top},\quad with \ \lambda\in\bZ.\]
Here we note that  $x$ is decided with $mod(v_1^4)$.  However
if $x=2^2\lambda v_1^3y_{top}+v_1^4y'$, then $|y'|>|y_{top}|$
and $(*)$ holds exactly in $k^*(\bar R(\bG))$.

Let us write $x''=c_2c_4c_6c_7$, i.e.,
\[x''=v_1y_6v_1y_{10}(2y_{12}+v_1y_{14})\cdot 2y_{14}=
2^2v_1^2y_{top}\quad mod(v_1^3).\]
Then  we have $x-\lambda v_1x''=0$
$mod(v_1^4)$.
Hence $x=0$ in $gr_{\gamma}^*(R(\bG))/2$.
\end{proof}
In the above proof, torsion index and elements 
 $e_8, e_{16}$ do not appear.  Hence 
the above proof is seen also the proof for $G=Spin(15)$.

We will disprove the Karpenko conjecture by using the algebraic integral
Morava $A\tilde K(n)$-theory, and
its associated graded ring 
$ gr(n)^*(X)=CH^*(X)/I(n).$  Recall the Nishimoto result
$ Q_i(x_3)=y_{2^{i+1}+2},$ and $ Q_i(x_5)=y_{2^{i+1}+4}.$
Hence,  in  $\Omega^*(\bar R(\bG))/I_{\infty}^2$, we have
\[ \begin{cases}
c_2=v_1y_6+v_2y_{10}+v_3y_{18}+... \\
c_3=2y_6+v_2y_{12}+v_3y_{20}+...
\end{cases} \]

Now  we consider $\tilde k(2)$-theory
with $\tilde k(2)^*=\bZ_{(2)}[v_2]$. (Note
the res $\tilde K(2)^*(R(\bG))\to \tilde K(2)^*(R(\bG))$ is not isomorphism, in general.) 
\begin{lemma}
Let $G=Spin(17)$ and $\bG$ be versal.
We have
\[ c_2c_3c_6c_7\not=0\quad in\ gr(2)^*(R(\bG))/2
\quad  (hence\ in\ CH^*(R(\bG))/2).\]
\end{lemma}
\begin{proof} 
We can write in $\tilde k(2)^*(\bar R(\bG))/(I_{\infty}^5).$
\[ x=pr(c_2c_3c_6c_7)=(v_2y_{10})(2y_6+v_2y_{12})(2y_{12}+v_2y')(2y_{14}+v_2y'') \]
for $y',y''\in \tilde k(2)^*\otimes P(y).$  Hence we have 
\[x=2^3v_2y_{top}\quad mod(v_2^2).
\]

Suppose that $x=0\in gr(2)^*(R(\bG))/2$.
 Then     there is $x'\in \tilde k(2)^*\otimes \Lambda(c)=\tilde k(2)^*\otimes
\Lambda(c_2,...,c_7)\otimes \bZ_{(2)}[e_{16}]$
such that we have
\[ x=v_2x'\quad or \quad x=2x'.\]

Suppose $x=v_2x'$. Then
$x'=2^3y_{top}$ $mod(v_2)$.
Therefore we see $t(G)\le 2^3$, and this contradicts
to $t(G)=2^4$.  

The case $x=2x'$ is proved similarly
using the Quillen operation $r_{\Delta_2}(v_2)=2$.
In fact 
$r_{\Delta_2}x'=2^3y_{top}$ also.  
(For $j\not =2$,
$r_{\Delta_2}(v_j)=0\ mod(I_{\infty}^2)$ and we can define this Quillen 
operation in  $\tilde k(2)^*(-)$.)
\end{proof}

{\bf Remark.}
For $G=Spin(15)$, we have $c_2c_3c_6c_7=0$ also
in $gr(2)^*(R(\bG))/2$. form $t(G)=2^3$.  Indeed
$e_{8}c_6c_7=2^3y_{top}$. For $G=Spin(13)$,
we have $c_2c_3=0$ also in $gr(2)^*(R(\bG))/2$.

\begin{thm} When $G=Spin(17)$ and $\bG$ be versal.
Then $I(1)\not =0$.
\end{thm}

Similarly, we have
\begin{lemma}
Let $G=Spin(19)$ and $\bG$ be versal.  For $x=c_2c_3c_6e_{16}$, we have $x \not =0 \in CH^*(R(\bG))/2$ but
$x=0\in gr_{\gamma}(R(\bG))/2$.
\end{lemma}
\begin{proof}
We note that from $Q_0(z)=y_{14}y_{18}$, we see 
$ e_{16}=2(y_{14}y_{18})$ in $CH^*(\bar R(\bG)),$
The torsion index $t(G)=2^4$ by Totaro
(in fact, $c_3c_5c_6e_{16}=2^4y_{top}$).
The results are gets by arguments exchanging 
$c_7$ by $e_{16}$ in the proofs of Lemma 10.2, 10.3 for $Spin(17)$.
The results are gets by arguments similar to the
preceding lemmas.
\end{proof}

\begin{cor}     Let $G=Spin(17)$ and $\bG$ be versal.
Then we have
\[ (2^4,2^3v_1,2^3v_2)\{y_{top}\}\subset Im(res_{\Omega})\subset
grBP^*(\bar R(\bG)).\]
\end{cor}
\begin{proof} 
We see
$ c_2c_3c_6c_7=2^3v_2y_{top}$  $mod(I_{\infty}^5).$
Hence the ideal $(2^4,2^3v_1,2^3v_2)\subset BP^*$
is invariant under $r_{\Delta_1},r_{2\Delta_1},$ and
$r_{\Delta_2}$, we have the corollary.  In fact
\[ c_3c_5c_6c_7\mapsto 2^4y_{top},\quad
c_2c_5c_6c_7\mapsto 2^3v_1y_{top},\quad 
c_2c_3c_6c_7\mapsto 2^3v_2y_{top}.\]
\end{proof}

\section{$Spin(N)$, $N$ ; large}

At first we see that fixed $*>0$, the Chow group 
$CH^*(R(\bG))$ is constant when $\ell$ becomes bigger.
\begin{thm}
Let $G(n)=Spin(n)$ and $\bG(n)$ be
the versal $G(n)_k$-torsor.  Then
we have the additive isomorphism
 \[ lim_{\gets N}CH^*(R(\bG(N)))/2\cong \Lambda(c_2,c_3,...,c_n,...). \]
\end{thm} 
\begin{proof}
Let $N=2\ell+1$,  and
$ 2^2\le 2^n<2^t\le \ell<2^{t+1}.$  We will  see 
\[  CH^*(R(\bG(N))/2\cong \Lambda (c_2,...,c_{2^n-1})
\quad for \  *< 2^n. \]

Suppose that
\[ x= \sum  c_{i_1}... c_{i_s} =0\in CH^*(R(\bG(N)))/2\quad
for \ \ 2\le  i_1<...<i_s<2^n.\]
Here $c_{i_j}=v_ny_m$ with $m={2^n-2+2i_j}$ in $k(n)^{*}(\bar R(\bG)/2$.  Since $2^n<m<2^{n+1}$,
$m$  is not a form $2^r$, $r>3$. Hence $y_m$
 is a generator of $grP(y)\cong \Lambda_{s\not =2^i}y_{2s}$.

Moreover  recall that  
\[ e_{2^{t+1}}=v_ny_{2^{t+1}-2+2^n-2}y_{2^{t}+2}+...\]
This element  is in the $ideal(v_n^2,E)$ with
$E=(y_{2j}|j> 2^{t}).$
Hence we see $c_{i_j}=v_ny_m$ is  also nonzero $mod(v_n^2,E)$ since $n<t$.

Thus we see that 
for $x'= y_{2^n-2+2{i_1}}....y_{2^n-2+2i_s}$
\[ x=\sum v_n^sx'\not =0 \in k(n)^*(\bar R(\bG))/2\cong k(n)^*\otimes P(y).\]
Moreover $v_n^{s-1}x'\not \in Im(res)$, because 
$Im(res)$ 
 is generated by $res(c_{j_1})...res(c_{j_r})$
and each $res(c_j)=0$ $mod(v_n)$.
\end{proof}

We note the following general properties for $\ell\ge 5$.
\begin{lemma}
Let $2^t<\ell<2^{t+1}$ with $\ell \ge 5$.
 Then
$e_{2^{t+1}}\not =0$
in $gr_{\gamma}^*(R(\bG))/2$.
\end{lemma}
\begin{proof}
At first, we recall the case $\ell=5$.  Then $e_8\not =0$
in $CH^*(R(\bG))/2$ from Theorem 7.1.
By the naturality $e_8\not =0$ for $\ell=6,7$. Moreover
we note
\[ 2^3=t(Spin(15))<t(Spin(19))=2^4\quad (while\ 
t(Spin(7))=2=t(Spin(11))).\]

Let us write $d=2^{t+1}$ and assume $d\ge 16$ (i.e., $t\ge 3$).
We consider the group $ G=Spin(d+3)$.
Since $Q_0(z)=y_{d-2}y_{d+2}$ and $d(z)=e_d$ 
by Nishimoto, we have
\[ 2y_{d-2}y_{d+2}=e_d  \quad \in CH(\bar R(\bG))/2.\]

Here we recall
\[( *)\quad t(Spin(d-1))<t(Spin(d+3)).\]
In fact, by Totaro, for 
$t(Spin(2n+1))=2^u$, $u$ is close to $n-2log_2(n)+1$. 
Then we can see that $e_d\not =0$ in  $CH^*(R(\bG))/2$. Otherwise $1/2e_d=y_{d-2}y_{d+2}
\in Res$.  This implies
        \[ t(Spin(d-1))y_{6}...y_{d-4}\cdot y_{d-2}y_{d+2}
\in Res(CH^*(R(\bG))),\]
which contradicts to $ (*)$.  Hence From
 Corollary 3.7,
we also see $e_d\not =0\in CH^*(R(\bG_k))/2$
for $\ell=d/2+1$.

For groups $Spin(2\ell+1)$, $d/2<\ell<d$, we see 
\[ d(z)=c_1^d=2(y_{d-2}y_{d+2}+y_{d-4}y_{d+4}+...)\]
We have $e_d\not=0$ for the group
$Spin(2\ell+1)$, by the naturality for
$Spin(d+3)\to Spin(2\ell+1)$.
\end{proof}


\end{document}